\documentclass[3p,number,times]{elsarticle}
\usepackage{natbib}
\usepackage{endnotes}
\let\footnote=\endnote

\usepackage{amsmath,amssymb,amsthm}

\begin{document}

\begin{frontmatter}

\title{Is the p-value a good measure of evidence? An asymptotic consistency~criterion} 

\author[mg]{M.~Grend\'ar}\fnref{mxf}
\ead{marian.grendar@savba.sk.}
\fntext[mxf]{Supported by VEGA grant 1/0077/09. Valuable feedback from George Judge, Lenka Mackovi\v cov\'a, J\'an Ma\v cutek, Jana Majerov\'a, Andrej P\'azman, Franti\v sek Rubl\'ik, Vladim\'ir \v Spitalsk\'y, Franti\v sek \v Stulajter and Viktor Witkovsk\'y is gratefully  acknowledged.}

\address[mg]{Department of Mathematics, FPV UMB, 974 01 Bansk\'a Bystrica, Slovakia; Institute of Mathematics and Computer Science of the Slovak Academy of Sciences (SAS), and UMB, Bansk\'a Bystrica, Slovakia; Institute of Measurement Sciences SAS, Bratislava, Slovakia}

\begin{abstract}
What are the criteria  that a measure of statistical evidence  should satisfy? 
It is argued that a measure of evidence should be consistent. Consistency is an asymptotic criterion:
the probability that if a measure of evidence in data strongly testifies against a hypothesis $H$, then 
$H$ is indeed not true, should go to one,  as more and more data appear.
The p-value is not consistent, while the ratio of likelihoods is.
\end{abstract}


\end{frontmatter}

\newtheorem{thm}{Proposition}
\newtheorem*{note}{Note}

\section{Introduction}

The p-value is commonly used as a measure of evidence in a data $X_1^n$, against a hypothesis $H_1$: the smaller the p-value, the stronger the evidence against $H_1$ in the data. Recall that the p-value is the smallest level at which a test $T(X_1^n)$ rejects $H_1$. According to the typical calibration \cite{CH}, \cite{Wasserman}, the p-value smaller than 0.01 suggests a very strong evidence against $H_1$.

Unlike the p-value, which measures evidence against a single hypothesis, the ratio of likelihoods\endnotemark[1] measures evidence in a data for a simple hypothesis $H_1$, relative to a simple hypothesis $H_2$. For a parametric model $f_X(x\,|\, \theta)$, the ratio of likelihoods $r_{12} = f(X_1^n\,|\, H_1)/f(X_1^n\,|\, H_2)$ measures evidence for $H_1$ relative to $H_2$, in data $X_1^n$. The value of $r_{12}$ above a certain threshold $k > 1$ is taken as an evidence in favor of $H_1$, and against $H_2$. Values of $k$ around~$30$ are suggested for a threshold, above which the evidence is considered very strong (cf. \cite{Royall_book}, \cite{Barnard}).

Statistics abounds criteria for assessing quality of estimators, tests, forecasting rules, classification algorithms, but besides the likelihood principle discussions (cf. \cite{BW}), it seems to be almost silent on what criteria should a good measure of evidence satisfy.
Schervish, in a notable exception\endnotemark[2] \cite{Schervish}, considers a requirement of coherence, borrowed from the multiple comparisons theory \cite{G}.  
If $H: \theta\in\Theta$ implies $H': \theta\in\Theta'$ (i.e., $\Theta\subset\Theta'$), 
then the coherent measure of evidence gives at least as strong evidence to $H'$ as it gives to $H$.
The p-value is not coherent; cf. \cite{Schervish}.
In this note, an asymptotic criterion of consistency is introduced, and it is demonstrated that the p-value is not consistent, while the ratio of likelihoods satisfies the consistency requirement.

\section{Measure of evidence}

To set a formal framework, let $X\in\mathrm{R}^K$
be a random variable with the probability density (or mass) function $f_X(x\,|\, \theta)$, parametrized by $\theta\in\Theta\subseteq\mathrm{R}^L$, and such that if $\theta\neq\theta'$ then $f_X(\cdot\,|\,\theta)\neq f_X(\cdot\,|\,\theta')$.
Let $\Theta_1, \Theta_2$ form a partition of $\Theta$, and associate $\Theta_j$ with the hypothesis $H_j$, $j = 1,2$.  Let $X_1^n \triangleq X_1, \dots, X_n \sim f_X(x\,|\, \theta)$ be a random sample from $f_X(x\,|\, \theta)$.
A measure of evidence $\epsilon(H_1, H_2, X_1^n)$, in data $X_1^n$, for the hypothesis $H_1: X_1^n \sim f_X(x\,|\, \theta)$ where $\theta\in\Theta_1$, relative to  $H_2: X_1^n \sim f_X(x\,|\, \theta)$ where $\theta\in\Theta_2$, is a mapping $\epsilon(H_1, H_2, X_1^n): \Theta_1 \times \Theta_2 \times (\mathrm{R}^K)^n \rightarrow \mathrm{R}$.  
It usually  goes with a calibration that partitions  values of $\epsilon(\cdot)$ into intervals, or categories.
In what follows, the interest will concentrate on the category $S$ of the most extreme values of the evidence measure $\epsilon(\cdot)$ that correspond to the strongest evidence.
Finally, the measure of evidence against a hypothesis $H_1$, relative to $H_2$, in data $X_1^n$, will be denoted $\epsilon(\neg H_1, H_2, X_1^n)$.

\section{Consistency requirement}

In \cite{SBB}, Sellke, Bayarri, and Berger stress that in applications of an evidence measure, data sets may come from either $H_1$ or $H_2$.
The authors illustrate this important point by an example of testing drugs $D_1, D_2, D_3, \dots$, for an illness, in a series of independent experiments. The measure of evidence applied to a data set from $i$-th experiment, is used to differentiate between the hypothesis $H_1$ that the drug $D_i$ has a negligible effect, and the alternative $H_2$ that the drug $D_i$ has a non-negligible effect. 
Some drugs have negligible effects, some have the non-negligible one. In other words, some experimental data $X_1^n$ come from $H_1$, other data sets are from $H_2$. This key aspect of applications of the evidence measure can be captured by the following two-level sampling mechanism:
\begin{enumerate}
\item First, $\theta$ is drawn from a pdf (or pmf) $p(\theta)$.

\item Given $\theta$, 
a random sample $X_1^n$ is drawn from $f_X(x\,|\, \theta)$.
\end{enumerate}

As the sample size $n$ increases, it should hold, informally put,  that among the data sets which, according to the measure of evidence strongly
testify against $H_1$, the relative number of those which in fact come from $H_1$, should go to zero. This motivates the following requirement of consistency\endnotemark[3]: We say that a measure of evidence $\epsilon(\neg H_1, H_2, X_1^n)$ against $H_1$, relative to $H_2$,  is \emph{consistent}, if
$$
\lim_{n\rightarrow\infty} Pr(H_1\,|\, \epsilon(\neg H_1, H_2, X_1^n) \in S) = 0.
$$

The probability that $\theta$ is in $\Theta_1$, given that the measure of evidence $\epsilon(\neg H_1, H_2, X_1^n)$ strongly testifies against $H_1$, relative to $H_2$, should go to zero, as the sample size $n$ goes beyond any limit.

\section{Is the p-value consistent?}

The p-value is $\pi \triangleq \inf\, \{\alpha: T(X_1^n)\in R_\alpha\}$, where $T$ is a test statistic, $\alpha$ is the size of the test, and $R_\alpha$ is the rejection region for $H_1$. 
In this section it is assumed that $X$ is a continuous random variable and the test statistic $T$ is such that it rejects $H_1$ when the observed value $t$ of $T$ is large. Then the p-value is $\pi = \sup_{\Theta_1} Pr(T > t\,|\, \theta)$. The p-value $\pi(\neg H_1,\cdot, X_1^n)$ as a measure of evidence against $H_1$ does not take $H_2$ into account. Let $S = [0, \alpha_S)$ be the interval of values that indicate the very strong evidence against $H_1$.

Before addressing the question of consistency of the p-value in general, consider an illustrative example of
the gaussian random variable $X$ with the variance $\sigma^2 = 1$, 
and let $\Theta_1 = \{\theta_1\}$, $\Theta_2 = \{\theta_1 + \delta\}$, $\delta > 0$. 
Let $w = p(\Theta_1)$, $w \in (0, 1)$.
And, let $T(X_1^n) = \sqrt{n}(\overline{x} - \theta_1)$ 
be the test statistic, and $R_\alpha = \{X_1^n: T(X_1^n) > z_{1-\alpha}\}$ be the rejection region, with $z_{1-\alpha}$ denoting the $1-\alpha$ quantile of the standard normal distribution.

Under $H_1$, the p-value is a uniform random variable, so $Pr(\pi(\neg H_1,\cdot, X_1^n) \in S\,|\, \Theta_1) = \alpha_s$. Under $H_2$, the power of the test is
$Pr(\pi(\neg H_1, \cdot, X_1^n) \in S\,|\, \Theta_2) = 1 - \Phi(z_{1-\alpha_s} - \sqrt{n}\delta)$, where $\Phi(\cdot)$ is the distribution function of the standard normal random variable. 
Note that $Pr(\pi(\neg H_1,\cdot, X_1^n) \in S\,|\, \Theta_2)$ converges to $1$, for $\delta > 0$.
Taken together,
$\lim_{n\rightarrow\infty} Pr(H_1 \,|\, \epsilon(\neg H_1,\cdot, X_1^n) \in S) =
    \frac{\alpha_S w}{1 - w(1-\alpha_S)}$.
Thus, in this simple example, the p-value is not a consistent measure of evidence against $H_1$.

Following the reasoning in the above example, it can be demonstrated that the p-value is inconsistent\endnotemark[4].

\begin{thm}
Let $\Theta_1$, $\Theta_2$ form a partition of $\Theta$.
Let $p(\theta)$ be such that $w \triangleq \int_{\Theta_1} p(\theta)$ is $w\in(0,1)$.
And, let  $T$, $R_\alpha$, be such that  $Pr(\pi(\neg H_1,\cdot, X_1^n) \in S\,|\, \Theta_2)\rightarrow 1$, as $n\rightarrow\infty$ (i.e., for $\theta\in\Theta_2$, the power of the test $T$ converges to $1$). Then it holds that
\begin{equation}\label{1}
  \lim_{n\rightarrow\infty} Pr(H_1 \,|\, \pi(\neg H_1,\cdot, X_1^n) \in S) = \frac{\alpha_S w}{1 - w(1 - \alpha_S)}.
\end{equation}
\end{thm}

\begin{proof} 
Under $H_1$, the p-value is uniformly distributed, so that $Pr(\pi(\neg H_1,\cdot, X_1^n) \in S)\,|\,\theta) = \alpha_S$, for $\theta\in\Theta_1$. Thus, $\int_{\Theta_1} Pr(\pi(\neg H_1,\cdot, X_1^n) \in S \,|\,\theta)p(\theta) = \alpha_S w$.
Next, under the assumption that the power of the test $T$ goes to $1$, as $n\rightarrow\infty$, the probability $\int_{\Theta_2} Pr(\pi(\neg H_1,\cdot, X_1^n) \in S \,|\,\theta)p(\theta) \rightarrow 1 - w$.  Taken together, it proves the Proposition.
\end{proof}

Since the right-hand side expression in \eqref{1} is positive, the p-value is not a consistent measure of evidence.
The limit of the probability becomes zero only at the extreme, uninteresting case of $w = 0$, i.e., when  no $X_1^n$ comes from $H_1$.  For 
the typical value of $\alpha_S = 0.01$ and $w = 1/2$, the limit value of the probability is $\alpha_S/(1+\alpha_S) = 0.0099$. For $w = 0.9$, the probability is $0.0826$. For $w = 0.999$, the probability is $0.9090$, and it converges to $1$, as $w\rightarrow 1$. The greater the relative presence of data sets from $H_1$, the higher the asymptotic probability that the data come from $H_1$, when  the p-value strongly testifies against $H_1$.

\section{Is the ratio of likelihoods consistent?}

For point sets $\Theta_1$, $\Theta_2$, the ratio of likelihoods $r_{12}$ of $H_1$ relative to $H_2$ is $r_{12} \triangleq f_1/f_2$, where $f_j \triangleq f_{X_1^n}(x_1^n\,|\, \Theta_j)$, for $j = 1,2$. The ratio $r_{12}$ measures the evidence in favor of $H_1$ (and against $H_2$), in data $X_1^n$. The larger the $r_{12}$, the stronger the evidence in favor of $H_1$ (and against $H_2$), so that $S = [k_S, \infty)$, $k_S > 1$.

First, consider the ratio of likelihoods $r_{21}$ in the example described above. 
Clearly, $Pr(r_{21}(\neg H_1, H_2, X_1^n) \in S \,|\, \Theta_1) = 1 - \Phi(\log k_S/\delta\sqrt{n}  + \sqrt{n}\delta/2)$, which, under the assumption $\delta > 0$, converges to $0$, as $n\rightarrow\infty$. And,
$Pr(r_{21}(\neg H_1, H_2, X_1^n) \in S \,|\, \Theta_2) = 1 - \Phi(\log k_S/\delta\sqrt{n}  - \sqrt{n}\delta/2)$, which, under the assumption $\delta > 0$, converges to $1$, as $n\rightarrow\infty$.
Thus, $\lim_{n\rightarrow\infty} Pr(H_1\,|\, r_{21}(\neg H_1, H_2, X_1^n) \in S) = 0$.
Hence, the ratio of likelihoods is a consistent measure of evidence, in this example.

And the consistency is not accidental, as stated in the following Proposition.

\begin{thm}
For point sets $\Theta_1$, $\Theta_2$, and $p(\Theta_1)  \in (0, 1)$,  
the ratio of likelihoods $r_{21}(\neg H_1, H_2, X_1^n)$ is a consistent measure of evidence, i.e.,
\begin{equation*}
\lim_{n\rightarrow\infty} Pr(H_1 \,|\, r_{21}(\neg H_1, H_2, X_1^n) \in S) = 0.
\end{equation*}
\end{thm}

\begin{proof} The claim follows from the Law of Large Numbers (LLN), applied to $1/n\log f_2/f_1\,|\,\Theta_j$, and the fact that the Kullback Leibler
divergence is positive for distinct distributions.
\end{proof}

Recently, Bickel \cite{Bickel} proposed an extension of the ratio of likelihoods (see also \cite{LS}, \cite{Zhang})
to the case of general $\Theta_1$, $\Theta_2$: $r_{12}^e \triangleq \sup_{\Theta_1} f(X_1^n\,|\, \theta)/{\sup_{\Theta_2} f(X_1^n\,|\, \theta)}$, and suggested its use as a measure of evidence.
The extended ratio of likelihoods reduces to the ratio of likelihoods, when $\Theta_1$, $\Theta_2$ are point sets.
Under additional assumptions, $r_{21}^e$ is a consistent measure of evidence. Before stating the result, recall that the maximum likelihood (ML) estimator $\hat\theta(\tilde\Theta)$ of $\theta$, restricted to $\tilde\Theta\subset\Theta$, is $\hat\theta(\tilde\Theta) \triangleq \arg\sup_{\theta\in\tilde\Theta} f_{X_1^n}(x_1^n\,|\,\theta)$.

\begin{thm}
Let $f_X(x\,|\, \theta)$ and $\Theta_1$, $\Theta_2$ be such that the maximum likelihood estimators $\hat\theta_j(\Theta_j)$, restricted to $\Theta_j$, are consistent estimators of $\theta$, $j = 1, 2$.
And let the maximum likelihood estimators
$\hat\theta_j(\Theta_i)$,  restricted to $\Theta_i$,
converge in probability to some finite $\bar\theta_j$, $i,j\in\{1,2\}, i\neq j$. Let $p(\theta)$ be such that $\int_{\Theta_1}
p(\theta) \in (0, 1)$.
Then the extended ratio of likelihoods $r_{21}^e(\neg H_1, H_2, X_1^n)$ is a consistent measure of evidence against $H_1$, relative to $H_2$.
\end{thm}

\begin{proof} 
Under the assumed consistency and convergence of the constrained MLs, the claim follows from the LLN and the positivity of the Kullback Leibler divergence between two different distributions, applied to the probability $Pr(r_{21}^e > k_S\,|\,\theta)$ in the 
upper bound $\frac{\left[\sup_{\Theta_1} Pr(r_{21}^e > k_S\,|\,\theta)\right]\int_{\Theta_1}p(\theta)}{\int_{\Theta_2} Pr(r_{21}^e > k_S\,|\,\theta)p(\theta)}$ and the lower bound $\left[\inf_{\Theta_1} Pr(r_{21}^e > k_S\,|\,\theta)\right]\int_{\Theta_1}p(\theta)$ of 
$Pr(\theta\in\Theta_1\,|\, r_{21}^e(\neg H_1, H_2, X_1^n)\in S)$.
\end{proof}

\section{Is the Bayes factor consistent?}

It is open to debate whether a measure of evidence can depend on a prior information. Bayesians usually measure evidence for $H_1$ relative to $H_2$ by the Bayes Factor $b_{12} \triangleq \int_{H_1} f(X_1^n\,|\,\theta)q(\theta)\,d\theta/\int_{H_2} f(X_1^n\,|\,\theta)q(\theta)\,d\theta$, where $q(\cdot)$ is the prior distribution. The Bayes Factor above $150$ is usually considered \cite{KR} as the very strong evidence for $H_1$. However, Lavine and Schervish \cite{LS} note that the Bayes factor does not satisfy the coherence requirement, while the posterior odds is coherent. Both the Bayes factor $b_{21}$ and the posterior odds $p_{21}(H_2, H_1, X_1^n) \triangleq b_{21}\,q(\Theta_2)/q(\Theta_1)$  are consistent measures of evidence against $H_1$, relative to $H_2$. Also, in analogy with the Proposition 3, consistency of the ratio of posterior modes can be established.

\section{Conclusions}

There are several measures of statistical evidence in use. Among them is the Fisherian p-value and its extensions, likelihood-based measures, such as the ratio of likelihoods and the extended ratio of likelihoods, as well as the Bayes factor and the posterior odds.
What are the criteria that a measure of evidence should satisfy? Coherence (cf. Sect. 1) is one such a criterion. It is a logical criterion. In this note, the asymptotic criterion of consistency was introduced. Besides being incoherent, the p-value is also inconsistent. The ratio of likelihoods and its extension are consistent and coherent measures. Among the Bayesian measures of evidence, for instance the posterior odds ratio is both coherent and consistent.

\theendnotes

\endnotetext[1]{Likelihood ratio is used in the Neyman Pearson hypothesis testing. To distinguish the evidential use of the likelihood ratio from its use in decision making, the former is referred to as the Ratio of Likelihoods (RL). RL has a rich history, cf. \cite{Fisher}, \cite{Barnard}, \cite{Hacking}, \cite{E}, \cite{Lindsey}, \cite{Lindsey1}, \cite{Royall_book}, \cite{Royall_1}, among others.}

\endnotetext[2]{See also Sect. 3.2 in Edward's monograph \cite{E}, and a recent work \cite{Bickel} of Bickel.}

\endnotetext[3]{In \cite{SBB}, Sellke, Bayarri, and Berger use a Monte Carlo simulation to estimate the probability $Pr(\Theta_1\,|\, \pi(\neg H_1; \cdot, X_1^n) \approx 0.05)$, for a point set $\Theta_1$, in small samples, for the p-value, and relate it to the analogous probability for the Bayes Factor, which is in the studied setting the same as the ratio of likelihoods. The authors do not propose an asymptotic  criterion for a measure of evidence.}

\endnotetext[4]{The Proposition 1 holds also for the p-value that is valid in the sense of Mudholkar and Chaubey \cite{MCh}.}

\bigskip
\rightline{to mar, in memoriam}

\end{document}